\documentclass[11pt,a4paper]{amsart}

\usepackage{amsmath}
\usepackage{amssymb}

\theoremstyle{plain}   
\begingroup 
\newtheorem*{th_CY}{Theorem CY}
\newtheorem{theorem}{Theorem}[section]   
\newtheorem{lemma}[theorem]{Lemma}         
\endgroup

\theoremstyle{definition}

\theoremstyle{remark}
\newtheorem{remark}[theorem]{Remark}        

\numberwithin{equation}{section}

\newcommand{\R}{{\mathbb R}}

\newcommand{\conv}{\operatorname{conv}}

\newcommand{\vf}{\varphi}

\newcommand{\diam}{\operatorname{diam}}

\usepackage[dvips]{color}

\begin{document}

\title{A note on propagation of singularities of semiconcave functions of two variables}

\thanks{The research was supported by the grant MSM 0021620839 from
 the Czech Ministry of Education and by the grant
GA\v CR 201/09/0067.}

\author{Lud\v ek Zaj\'\i\v{c}ek}



\email{zajicek@karlin.mff.cuni.cz}

\address{Charles University,
Faculty of Mathematics and Physics,
Sokolovsk\'a 83,
186 75 Praha 8-Karl\'\i n,
Czech Republic}


\begin{abstract} 
P. Albano and P. Cannarsa proved in 1999 that, under some applicable conditions, singularities of semiconcave functions in $\R^n$ propagate along Lipschitz arcs. Further regularity properties of these arcs were proved by P. Cannarsa and Y. Yu in 2009.
We prove that, for $n=2$, these arcs are very regular: they can be found in the form (in a suitable Cartesian coordinate system) $\psi(x) = (x, y_1(x)-y_2(x)), \ x \in [0,\alpha]$, where $y_1$, $y_2$ are convex and Lipschitz on $[0,\alpha]$. In other words: singularities propagate along arcs with finite turn.
\end{abstract}

\maketitle

\markboth{L.~Zaj\'{\i}\v{c}ek}{Propagation of singularities}

\section{Introduction}

Let $u$ be a function defined on an open set $\Omega \subset \R^n$ which is locally (linearly) semiconcave; i.e.,
 $f$ is locally representable in the form  $u(x)= g(x) + K \|x\|^2$, where $g$ is concave (cf. \cite{CaSi}).
  
  Let $\Sigma(u)$ be the singular set of $u$, i.e.
  $$  \Sigma(u) = \{x \in \Omega:\ u \ \ \text{is not differentiable at}\ \ x\}.$$
  It is clear that in many questions concerning $\Sigma(u)$ we can suppose that $u$ is concave (or convex), since
   the results for semiconcave functions then easily follow. But it is reasonable to formulate theorems for semiconcave functions, since these functions are important in a number of applications (see \cite{CaSi}).
   
   It is well-known that $\Sigma(u)$ is a rather small set: it can be covered by countably many Lipschitz
    DC hypersurfaces (\cite{zajkonv}). (Note that for $A \subset \R^n$ there exists a convex (resp. semiconcave)
     function $u$ on $\R^n$ such that $A = \Sigma(u)$, if and only if $A$ is an $F_{\sigma}$ set which
can be covered by countably many Lipschitz   DC hypersurfaces, see \cite{Pa}.)

The set $\Sigma(u)$  can have isolated points, but P. Albano and P. Cannarsa \cite{AC} found applicable conditions ensuring
 that $\Sigma(u)$ is in a sense big in each neigbourhood of a given $x_0 \in \Sigma(u)$. (The results of  \cite{AC}
  can be found also in the book \cite{CaSi}.)
 In particular, they proved that if $\partial D^+ u(x_0) \setminus D^* u(x_0) \neq \emptyset$ (see Preliminaries for
  the definitions), then a Lipschitz arc $\xi: [0,\tau] \to \Omega$ emanating from $x_0$ is a subset of the singular set  $\Sigma(u)$. 
The results of \cite{AC} were refined in \cite{CY}; in particular it is proved in \cite[Corollary 4.3]{CY} that  $\xi$ 
 has nonzero (right continuous) right derivative at all points.
 
 The purpose of the present note is to show that in $\R^2$ the results of \cite{CY} and  methods from \cite{zajkonv}
  and \cite{VeZa} easily imply that the restriction of $\xi$ to an interval $[0,\tau']$ has an equivalent parametrization of the form (in a suitable Cartesian coordinate system) $\psi(x) = (x, y_1(x)-y_2(x)), \ x \in [0,\alpha]$, where $y_1$, $y_2$ are convex and Lipschitz on $[0,\alpha]$.  (This result is equivalent to the assertion that the  restriction of  $\xi$ to an interval $[0,\tau^*]$ has finite turn, cf. Remark \ref{2}).
   In particular, $\xi$ has (left continuous) left halftangents at all points. 
   
   The question whether the results  can be generalized to the case
   $n >2$ remains open.

\section{Preliminaries}

By $B(x,r)$ we denote the open ball with center $x$ and radius $r$. The scalar product of $v,w \in \R^n$ is denoted by
 $\langle v,w\rangle$. If $A \subset \R^n$, $c\in \R$ and $v \in \R^n$, then we define the sets $A+v$ and $c A$ by
  the usual way and similarly set $\langle v,A\rangle := \{\langle v,a\rangle:\ a \in A\}$. 
   The boundary and the convex hull of a set $A \subset \R^n$ are denoted by $\partial A$ and $\conv A$, respectively.
  The (Fr\' echet) derivative $Df(a)$ of
   a function $f$ on $\R^n$ at $a \in \R^n$ is considered as an element of $\R^n$. The one-sided derivatives of a real or vector function $\xi$ of one variable at $x \in \R$ are denoted by $\xi'_+(x)$ and $\xi'_-(x)$.

  If $f$ is a  function defined on a subset of $\R^n$, $x \in \R^n$ and $v\in \R^n$, then we define the one-sided
   directional derivative as  
 $$f'_+(x,v):=\lim_{h\to 0+}\frac{f(x+hv)-f(x)}{h}.$$
   
   Let $\Omega \subset \R^n$ be an open set and $u$ a locally semiconcave function on $\Omega$ (see Introduction).
    Then $u$ is locally Lipschitz and so differentiable a.e. in $\Omega$. For $x \in \Omega$, we define (see \cite{AC} or \cite[p. 54]{CaSi})
     the set
     $$ D^*u(x) = \{p \in \R^n:\ \Omega \ni x_i \to x,\, Du(x_i) \to p\}$$
     of all {\it reachable gradients} of $u$ at $x$ (note that $D^*u(x) $ is also called limiting subdifferential, cf. \cite[p. 725]{AC}).
     
     The {\it superdifferential} $D^+u(x)$ of $u$ at $x$ can be defined as the convex hull of $ D^*u(x)$ (see \cite[p. 723]{AC}, cf. \cite[Theorem 3.3.6]{CaSi}).
     
     Always  $D^*u(x) \subset \partial D^+u(x)$ (see \cite[Proposition 3.3.4]{CaSi}).
     Note that the superdifferential $D^+u(x)= \conv  D^*u(x)$ coincides with the Clarke's subdifferential $\partial^C u(x)$ (since $\partial^C u(x)= \conv  D^*u(x)$, see, e.g., \cite{Cl}).

Let $u(x)= g(x) + K \|x\|^2$, where $g$ is concave, on a ball $B(x_0,\delta) \subset \Omega$. Set $f:=-g$. Since $D(K\|x\|^2) = 2Kx$, we easily obtain that
$ D^*u(x_0)= -D^*f(x_0)+ 2Kx_0$, and therefore
\begin{equation}\label{redconv}
D^+u(x_0) = -\partial f(x_0) + 2Kx_0,
\end{equation}
where $\partial f$ is the classical subdifferential of the convex function $f$. 

Recall that a function defined on an open convex subset of $\R^n$ is a {\it DC function} if it is a difference of two convex functions.
 We will need the following simple lemma which is a special case of the ``mixing lemma'' \cite[Lemma 4.8]{VeZa}.
 \begin{lemma}\label{mix}
 Let $\vf_1,\dots,\vf_p$ be DC functions on $\R$,  and let $h$ be a continuous function on $\R$ such that
 $$h(x) \in \{\vf_1(x),\dots,\vf_p(x)\}\ \ \ \text{for each}\ \ \ x \in \R.$$
 Then $h$ is DC on $\R$.
 \end{lemma}
 We will need also the well-known fact that convex functions are semismooth (see \cite[Proposition 3]{Mi}, cf. also \cite[Proposition 2.3]{Sp}). In other words:
 \begin{lemma}\label{semismooth}
 Let $f$ be a convex function on an open convex set $C \subset \R^n$ and $x_0 \in C$. Let $0 \neq q \in \R^n$, $q_n \to q$, $t_n \searrow 0$, and $z_n \in \partial f(x_n)$, where $x_n:= x_0 +t_nq_n$, be given. Then
  $\langle q,z_n\rangle \to f_+'(x_0,q)$. In particular,
  \begin{equation}\label{diamknule}
  \diam \langle q, \partial f(x_n)\rangle \to 0.
  \end{equation}  
 \end{lemma}
 
\section{The result and its proof}
The following result is an immediate consequence of \cite[Corollary 4.3]{CY}.
\begin{th_CY}\label{ACY}
Let $u$ be a semiconcave function on an open set $\Omega \subset \R^n$, $x_0 \in \Sigma(u)$ be a singular point of $u$ and
$$\partial D^+ u(x_0) \setminus D^* u(x_0) \neq \emptyset.$$
Then there exists $q \in \R^n$ with $\|q\|=1$, $\tau>0$, and a Lipschitz curve $\xi: [0,\tau] \to \Sigma(u)$
 such that
 \begin{enumerate}
 \item $\xi'_+(0) =q$,
 \item $\lim_{s \to 0+} \xi'_+(s) = q$, and
 \item $\inf_{s \in [0,\tau]}\ \diam\,D^+ u(\xi(s)) \, > 0$.
 \end{enumerate}
\end{th_CY}
Note that it is proved in \cite{CY} also that  $\xi'_+(s)$ exists for each $s \in [0,\tau)$ and $\xi'_+$ is right continuous on $[0,\tau)$. 
Further note that the result without (ii) was proved already in \cite{AC}.

Using Theorem CY and the method of the proof of the implicit function theorem for DC functions
 \cite[Theorem 4.4]{VeZa}, we easily prove the following result.

 \begin{theorem}\label{main}
 Let $u$ be a semiconcave function on an open set $\Omega \subset \R^2$, $x_0 \in \Sigma(u)$ be a singular point of $u$ and
$$\partial D^+ u(x_0) \setminus D^* u(x_0) \neq \emptyset.$$
Then there exist a Cartesian coordinate system in $\R^2$ given by a map $A:\R^2\to \R^2$ such that $A(x_0)=(0,0)$, and convex Lipschitz functions $y_1, y_2$ on some $[0,\alpha]$ ($\alpha>0$) such that, denoting $\psi(x):= (x, y_1(x)-y_2(x))$, $x \in [0,\alpha]$, we have $\psi(0) = (0,0)$ and $A^{-1}(\psi([0,\alpha])) \subset \Sigma(u)$.
\end{theorem}
\begin{proof}
Let $\xi: [0,\tau] \to \Sigma(u)$  and $q \in \R^2$ have properties from  Theorem CY. 
  We will proceed in four steps. In steps 1-3 we will suppose that 
  \begin{equation}\label{spec}
 x_0 = (0,0)\ \ \text{and}
 \ \ \ q= (1,0).
 \end{equation}
\smallskip

{\it Step 1}\ \  Set $e_2:= (0,1)$.
 Let $u(x)= g(x) + K \|x\|^2$ for $x \in B(x_0,\delta) \subset \Omega$, where $g$ is concave and Lipschitz with a constant $L>0$ on  $B(x_0,\delta)$. 
 Set $f:=-g$. Applying \eqref{redconv} to any point $x \in B(x_0,\delta)$, we obtain 
$D^+u(x) = -\partial f(x) + 2Kx, \ x \in B(x_0,\delta)$. So (iii) (of Theorem CY) easily implies that, for some $0 <\tau_1< \tau$,
 we have that  $f(\xi(s)) \in B(x_0,\delta)$ and $\partial f(\xi(s)) \subset B(0,L)$ for each $s \in [0,\tau_1]$, and
\begin{equation}\label{pdvel}
\inf_{s \in [0,\tau_1]}\ \diam\, \partial f(\xi(s)) \, > 0.
\end{equation}

We will show that there exists $0< \tau_2 < \tau_1$ such that
 \begin{equation}\label{diam1}
 \delta := \inf_{s \in (0,\tau_2]}\ \diam\,\langle e_2, \partial f(\xi(s))\rangle \, > 0
 \end{equation}
 Suppose on the contrary that there exits a sequence $(t_n)$ such that $t_n\searrow 0$ and 
 \begin{equation}\label{diam2}
 \lim_{n \to \infty}
  \diam\,\langle e_2, \partial f(\xi(t_n))\rangle = 0. 
  \end{equation}
   Set $q_n := \xi(t_n)/t_n$ and $x_n := \xi(t_n) = t_nq_n$. Since $q_n \to q$ by (i), Lemma \ref{semismooth}
    gives that
    \begin{equation}\label{diamx}
    \lim_{n \to \infty}
  \diam\,\langle q, \partial f(\xi(t_n))\rangle = 0. 
  \end{equation}
  Since \eqref{diam2} and \eqref{diamx} clearly imply $\lim_{n \to \infty}
  \diam\,  \partial f(\xi(t_n)) = 0$, we obtain a contradiction with  \eqref{pdvel}. 
  \medskip
  
{\it Step 2}\ \ Let $\xi = (\xi_1,\xi_2)$. By (ii), we have $\lim_{s\to 0+} (\xi_1)_+'(s) = 1$ and therefore
  there exits $0 < \tau_3 <\tau_2$ such that  $1/2 \leq (\xi_1)'(s)$ for a.e. $s \in (0,\tau_3)$. So
   $\xi_1$ is Lipschitz strictly   increasing on $[0,\tau_3]$ and $(\xi_1)^{-1}$ is Lipschitz on $[0, \alpha]$, where
    $\alpha:=\xi_1(\tau_3)$. Set  $g(x):= \xi_2 \circ (\xi_1)^{-1}(x),\ x \in [0,\alpha]$. Then 
     $g$ is Lipschitz and $\psi(x): = (x,g(x)),\  x \in [0,\alpha]$, is an equivalent parametrization of $\xi|_{[0, \tau_3]}$.
  \medskip
     
     {\it Step 3}\ \
     Choose a partition $-L=y_0 <y_1 <\dots<y_p=L\}$ of the interval $[-L,L]$ such that
      $\max\{y_{i} - y_{i-1},\ i=1,\dots,p\} < \delta/2$. For each $x \in (0,\alpha)$, the set
      $\langle e_2, \partial f(\psi(x))\rangle  \subset [-L,L]$ is a closed interval of length at least $\delta$ and so we can choose
      $i_x \in \{1,\dots,p\}$ such that
      \begin{equation}\label{oba}
      y_{i_x} \in \langle e_2, \partial f(\psi(x))\rangle\ \ \ \text{and}\ \ \ y_{i_x-1} \in \langle e_2, \partial f(\psi(x))\rangle.
      \end{equation}
      For $i\in  \{1,\dots,p\}$, set $A_i := \{x\in (0,\alpha):\ i_x = i\}$. We will show that, for each
       $i \in \{1,\dots,p\}$ with $A_i \neq \emptyset$, the function $g|_{A_i}$ can be extended to a Lipschitz DC function $\vf_i$ on $\R$.
       
       To this end, fix a such $i$ and set
       $$ \omega_1(x):= f(x,g(x))-y_i g(x)\ \ \text{and}\ \ \omega_2(x):= f(x,g(x))-y_{i-1} g(x)\ \ \text{for}\ x \in 
A_i.$$
 Since\ $\omega_1(x) - \omega_2(x) = (y_{i-1}- y_i) g(x),\ x \in A_i$, it is sufficient to prove that $\omega_i$
  ($i=1,2$) can be extended to a Lipschitz convex function $c_i$ defined on $\R$. 
  
  For each $x\in A_i$, choose $p_x \in \R$ such that $(p_x,y_i) \in \partial f (x,g(x))$ and consider the affine
   function
   $$ a_x(t) := \omega_1(x) + p_x(t-x),\ \ t \in \R.$$
   Set 
   $$c_1(t) := \sup \{a_x(t): \ x \in A_i\},\ \ t \in \R.$$
   Since $\omega_1$ is clearly bounded on $A_i$ and $|p_x| \leq L$ for $x\in A_i$, it is easy to see that
    $c_1$ is a Lipschitz convex function on $\R$. 
    
    Now consider arbitrary $x, t \in A_i,\ x\neq t$. Since $(p_x,y_i) \in \partial f (x,g(x))$, we have
    $$ f(t,g(t))-f(x,g(x)) \geq p_x(t-x) + y_i(g(t)-g(x)),$$
    and therefore
    $$\omega_1(t) = f(t,g(t))- y_i g(t) \geq f(x,g(x)) - y_i g(x) + p_x(t-x) = a_x(t).$$
    Since $a_t(t) = \omega_1(t),\ t \in A_i$, we obtain that $c_1$ extends $\omega_1$. Quite similarly
     we can find a convex Lipschitz extension $c_2$ of $\omega_2$.
     
     Since $g(x) \in \{\vf_1(x),\dots,\vf_p(x)\}$ for each $x \in (0,\alpha)$, and  $g,\ \vf_1,\dots,\vf_p$ are  continuous on $[0,\alpha]$, we can clearly find  $i_0, i_{\alpha} \in \{1,\dots,p\}$ such
      that $g(0) = \vf_{i_0}(0)$ and $g(\alpha) = \vf_{i_{\alpha}}(\alpha)$. 
      
      Let $h$ be the extension of $g$ with $h(x) = \vf_{i_0}(x),\ x<0$ and $h(x) =\vf_{i_{\alpha}}(x), x>\alpha$.
      Then $h$ is continuous on $\R$ and $h(x) \in \{\vf_1(x),\dots,\vf_p(x)\}$ for each $x \in \R$.
       Thus Lemma \ref{mix} implies that $h$ is DC on $\R$, i.e., $h= \gamma_1-\gamma_2$, where  $\gamma_1$ and $\gamma_2$ are convex on $\R$. Then  $y_i := \gamma_i|_{[0,\alpha]}$, $i=1,2$,  are clearly convex Lipschitz functions, and
        $\psi(x)= (x,y_1(x)-y_2(x)),\ x \in [0,\alpha]$.
\medskip

{\it Step 4}\ \ If \eqref{spec} does not hold, we can choose a Cartesian system of coordinates given by a map 
 $A: \R^2 \to \R^2$ such that $A(x_0) = (0,0)$ and $A(q)= (1,0)$. Applying steps 1-3 to $u^* := u \circ A^{-1}$ and
  $\xi^*:= A \circ \xi$, we obtain $\psi$ of the demanded form with $\psi([0,\alpha]) \subset \Sigma(u^*) = A(\Sigma(u))$.

      \end{proof}

      \begin{remark}\label{1}
      Well-known elementary properties of convex functions on $\R$ easily imply that the one-sided derivative
       $\psi_+'$ ($\psi_-'$) exists and is right (left) continuous on $[0,\alpha)$ ($(0,\alpha])$ and has finite variation on this interval. In other words, $\psi$ has {\it bounded convexity} (see \cite[Theorem 3.1]{VZ} or
        \cite[Lemma 5.5]{Du}). Further, since clearly $|\psi'_+|\geq 1,\ |\psi'_-|\geq 1$ we obtain that the
         curve $\psi$ has {\it finite turn} (see \cite[Theorem 5.4.2]{AR} or \cite[Theorem 5.11]{Du}). So the curve 
           $\psi^*: = A^{-1} \circ \psi$, for which $\psi^*([0,\alpha]) \subset \Sigma(u)$, 
             has also bounded convexity and finite turn.
            \end{remark}
            
         \begin{remark}\label{2}
          The proof of Theorem \ref{main} and Remark \ref{1} show that, for the curve $\xi: [0,\tau] \to \Sigma(u)$ from
           Theorem CY, there exists $0<\tau^*<\tau$ such that $\xi|_{[0,\tau^*]}$ has finite turn. In fact, this assertion ``is not weaker'' than Theorem \ref{main}, since it implies quickly by standard methods Theorem \ref{main}. 
         \end{remark}

          \begin{remark}\label{3}
          We {\it did not shown} that the curve $\xi$ from  Theorem CY has near $0$ (left-continuous) left derivative $\xi'_-$ at all points. However, the proof of Theorem \ref{main} clearly implies that $\xi$ has (left-continuous) left half-tangent on  $(0,\tau^*]$ for some $0<\tau^*<\tau$.
          \end{remark}
          
          We will not give detailed proofs of facts from Remarks \ref{1}-\ref{3}, since they would be inadequately long,  and these facts are not essential for the present short note.


\end{document}